\documentclass{article}
\usepackage{amssymb}
\usepackage{amsthm}
\usepackage{latexsym}
\usepackage{makeidx}
\makeindex 

\newcommand{\N}{{\mathbb N}}

\newcommand{\Z}{{\mathbb Z}}

\newcommand{\bx}{\hfill{$\Box $ }}
\newcommand{\Img}{\textrm{Im}\;}
\newcommand{\mod}{\textrm{mod}\;}
\newtheorem{lemma}{Lemma}[section]
\newtheorem{corollary}{Corollary}[section]
\newtheorem{proposition}{Proposition}[section]
\newtheorem{theorem}{Theorem}[section]
\newtheorem{remark}{Remark}[section]

\newtheorem{example}{Example}[section]
\title{Some remarks on Euler's totient function}
\author{R.Coleman\\Laboratoire Jean Kuntzmann\\Universit\'e de Grenoble}
\date{}
\begin{document}
\maketitle

\begin{abstract} The image of Euler's totient function is composed of the number 1 and even numbers. However, many even numbers are not in the image. We consider the problem of finding those even numbers which are in the image and those which are not.

If an even number is in the image, then its preimage can have at most half its elements odd. However, it may contain only even numbers. We consider the structure of the preimage of certain numbers in the image of the totient function.

Classification: 11A25, 11A41.
\end{abstract}

Euler's totient function $\phi$ is the function defined on the positive natural numbers $\N^*$ in the following way: if $n\in \N^*$, then $\phi (n)$ is the cardinal of the set 
$$
\{x\in \N^*:1\leq x\leq n, \gcd (x,n)=1\}.
$$
Thus $\phi (1)=1$, $\phi (2)=1$, $\phi (3)=2$, $\phi (4)=2$, and so on. The aim of this article is to study certain aspects of the image of the function $\phi$.\\

\section{Elementary properties} Clearly $\phi (p)=p-1$, for any prime number 
$p$ and, more generally, if $\alpha \in \N^*$, then $\phi (p^{\alpha})=p^{\alpha }- p^{\alpha -1}$. This follows from the fact that the only numbers which are not coprime with $p^{\alpha}$ are multiples of $p$ and there are $p^{\alpha -1}$ such multiples $x$ with $1\leq x\leq p^{\alpha}$.\\

It is well-known that $\phi$ is multiplicative, i.e., if $m$ and $n$ are 
coprime, then $\phi (mn) = \phi (m)\phi (n)$. If $n\geq 3$ and the prime 
decomposition of $n$ is 
$$
n = p_1^{\alpha _1}\ldots  p_s^{\alpha _s},
$$
then from what we have seen 
$$
\phi (n) = \prod _{i=1}^s(p_i^{\alpha _i}-p^{\alpha _i-1}) = 
n\prod _{i=1}^s(1-\frac{1}{p_i}).
$$
Notice that 
$$
p^{\alpha }- p^{\alpha -1}=p^{\alpha -1}(p-1).
$$
This implies that, if $p$ is odd or $p=2$ and $\alpha >1$, then 
$p^{\alpha }- p^{\alpha -1}$ is even. Hence, for $n\geq 3$, $\phi (n)$ is even.
Thus the image of $\phi$ is composed of the number $1$ and even numbers.\\

The following property is simple but very useful.

\begin{proposition}\label{prop1} If $p,m\in \N^*$, with $p$ prime, and $n=pm$, then $\phi (n)=(p-1)\phi (m)$, if $\gcd (p,m)=1$, and $\phi (n)=p\phi (m)$, if $\gcd (p,m)\neq 1$.
\end{proposition}

\begin{proof} If $\gcd (p,m)=1$, then we have
$$
\phi (n) =\phi (p)\phi (m) = (p-1)\phi (m).
$$
Now suppose that $\gcd (p,m)\neq 1$. We may write $m=p^{\alpha}m'$, with 
$\alpha \geq 1$ and $\gcd (p,m')=1$. Thus 
$$
\phi (n) = \phi (p^{\alpha +1})\phi (m') = p^{\alpha}(p-1)\phi (m').
$$
However 
$$
\phi (m) = p^{\alpha -1}(p-1)\phi (m')
$$
and so $\phi (n)= p\phi (m)$.\bx
\end{proof}

\begin{corollary} \label{cor1a} $\phi (2m)=\phi (m)$, if and
only if $m$ is odd.
\end{corollary}

\begin{proof} If $m$ is odd, then $\gcd (2,m)=1$ and so $\phi
(2m)=(2-1)\phi (m)=\phi (m)$. If $m$ is even, then $\gcd (2,m)\neq 1$, hence $\phi (2m) = 2\phi (m)\neq \phi (m)$. This ends the proof.
\end{proof}

\section{Bounds on $\phi ^{-1}(m)$} Let $m\in \N^*$ and consider the preimage $\phi ^{-1}(m)$ of $m$, i.e.
$$
\phi ^{-1}(m) = \{n\in \N^*:\phi (n)=m\}.
$$
There are several questions we might ask. First, is the set $\phi ^{-1}(m)$
empty and, if not, is it finite. It is easy to see that $\phi ^{-1}(1)=\{1,2\}$ 
and as we have already seen, $\phi ^{-1}(m)$ is empty if $m$ is odd and $m\neq 1$. It remains to consider the case where $m$ is even. The following result, due to Gupta \cite{gupta}, helps us to answer these questions.
 
\begin{proposition} Suppose that $m$ is an even number and let us set
$$
A(m) = m\prod _{p-1|m}\frac{p}{p-1},
$$
where $p$ is prime. If $n\in \phi ^{-1}(m)$, then $m<n\leq A(m)$. 
\end{proposition}

\begin{proof} Clearly, if $\phi (n)=m$, then $m<n$. On the other
hand, if $\phi (n) =m$ and $n = p_1^{\alpha _1}\ldots  p_s^{\alpha _s}$, then
$$
m = n\prod _{i=1}^s\frac{p_i-1}{p_i} \Longrightarrow 
	n=m\prod _{i=1}^s\frac{p_i}{p_i-1}.
$$
However, if $p|n$, then from the first section we know that $p-1|\phi (n)$ and it follows that, for each $p_i$, $p_i-1|m$. Hence $n\leq A(m)$.
\end{proof}

The proposition shows that the preimage of an element $m\in \N^*$ is always
finite. It also enables us to determine whether a given number
$m$ is in the image of $\phi$: we may determine $A(m)$ and then
calculate $\phi (n)$ for all integers $n$ in the interval $(m,A(m)]$. However, it should be noticed that the calculation of $A(m)$ depends on a knowledge of the factorization of $m$, which for large numbers may prove from a practical point of view difficult or even possible.\\

\begin{example} The divisors of $4$ are 1, 2 and 4. Adding 1 
to each of these numbers we obtain 2, 3 and 5, all of which are prime numbers. 
Thus 
$A(4) = 4\cdot\frac{2}{1}\cdot\frac{3}{2}\cdot\frac{5}{4} = 15$.
To find the inverse image of 4, it is sufficient to consider numbers between 5
and 15. In fact, $\phi ^{-1}(4)=\{5, 8, 10, 12\}$.
\end{example}

\begin{example} The divisors of 14 are 1, 2, 7 and 14. However, if we add 1 to each of these numbers we only find a prime number in the first two cases. Thus
$A(14)=14\cdot\frac{2}{1}\cdot\frac{3}{2}=42$. If we consider the numbers $n$ between 15 and 42, we find $\phi (n)\neq 14$ and so $\phi ^{-1}(14)=\emptyset$.
\end{example}

\begin{remark} The example $m=14$ shows that there are even
numbers which are not in the image of $\phi$.
\end{remark}

When deciding whether a number $m$ is in the image of $\phi$, it is not necessary to consider all numbers between $m$ and $A(m)$. In fact we only need to consider those numbers $n$ such that, for any prime number $p$, we have 
$$
p|n \Longrightarrow p-1|m.
$$
 
$A$ is a function defined on $\{1\}\cup 2N^*$. Let us look at the first values
of $A$ with the corresponding values of $\phi$ (when defined):
\begin{center}
\begin{tabular}{|c|c|c|}\hline
$m$ & $A(m)$ & $\phi (A(m))$\\ \hline
$1$ & $2$    & $1$ \\ \hline	 
$2$ & $6$    & $2$ \\ \hline
$4$ & $15$    & $8$ \\ \hline
$6$ & $21$    & $12$ \\ \hline
$8$ & $30$    & $8$ \\ \hline
$10$ & $33$    & $20$ \\ \hline
$12$ & $\frac{455}{8}$    & $-$ \\ \hline
$14$ & $42$    & $20$ \\ \hline
\end{tabular}
\end{center}
\vspace{3mm}
From the table we see that $A(m)$ may be odd, even or a fraction and that $A$ is
not increasing ($A(12)>A(14)$). Also, we may have $\phi (A(m))=m$ or 
$\phi (A(m))>m$.\\

It is interesting to consider the special case $m=2^k$. The only divisors of
$2^k$ are $1,2,2^2 \ldots ,2^k$. If we add 1 to each of these numbers, we obtain
numbers of the form $2^i+1$, where $0\leq i\leq k$. For $i\geq 1$, such
numbers are prime only if $i$ is a power of $2$. A number of the form
$F_n=2^{2^n}+1$ is said to be a Fermat number. Therefore
$$
A(2^k) = 2^k\cdot 2\cdot \prod \frac{F_n}{F_n-1}=
	2^{k+1}\cdot\prod \frac{F_n}{F_n-1},
$$
where the product is taken over Fermat numbers $F_n$ which are prime and such 
that $F_n-1|2^k$. For example, if $m=2^5=32$, then $F_0$, $F_1$ and $F_2$ are
the only Fermat numbers $F_n$ such that $F_n-1|m$. In addition these Fermat
numbers are all prime. Therefore 
$
A(m) = 64\cdot \frac{3}{2}\frac{5}{4}\cdot \frac{17}{16} = \frac{255}{2}
$. The Fermat numbers $F_0, \ldots , F_4$ are all prime; however $F_5,\ldots ,F_{12}$ are 
composite numbers and it has been shown that for many
other numbers $n$, $F_n$ is composite. In fact, up till now no Fermat number
$F_n$ with $n\geq 5$ has been found to be prime. If there are no prime Fermat
numbers with $n\geq 5$, then for $2^k\geq 2^{16}$ we have 
$$
A(2^k)=2^{k+1}\cdot \frac{F_0}{2}\cdot\frac{F_1}{2^2}\cdot\frac{F_2}{2^4}
	\cdot\frac{F_3}{2^8}\cdot\frac{F_4}{2^{16}}=2^{k-30}F_0F_1F_2F_3F_4, 
$$
which is an integer for $k\geq 30$.\\

Before closing this section, let us consider the upper bound on odd elements of
$\phi ^{-1}(m)$. From Corollary \ref{cor1a}, we know that if $n$ is odd and $n\in
\phi ^{-1}(m)$, then $2n\in \phi ^{-1}(m)$, therefore an upper bound on odd
elements of $\phi ^{-1}(m)$ is $\frac{A(m)}{2}$. We should also notice that at
least half of the elements in $\phi ^{-1}(m)$ are even. In fact, $\phi ^{-1}(m)$
may be non-empty and contain very few odd numbers, or even none. For example, 
the only odd number in $\phi ^{-1}(8)$ is $15$ and $\phi ^{-1}(2^{32})$ contains
only even numbers (see Theorem \ref{th5a} further on).

\section{The case $m=2p$} We now consider in some detail the case where $p$ is
prime and $m=2p$.

\begin{theorem} \label{thm3a} If $p$ is a prime number, then $2p$ lies in the image of $\phi$ if and only if $2p+1$ is prime.
\end{theorem} 

\begin{proof} If $2p+1$ is prime, then $\phi (2p+1)=2p$ and so 
$2p\in \Img\phi$.

Suppose now that $2p\in \Img\phi$. If $p=2$, then $2p=4\in \Img\phi$,
car $\phi (5)=4$ and $2p+1=5$, which is prime. Now suppose that $p$ is an odd
prime. As $2p\in\Img \phi$, there exists $n$ such that $\phi (n)=2p$. If $n=2^k$, then $2^{k-1}=\phi (2^k)=2p$, which is clearly impossible, because $p$ is an odd number greater than 1. Hence there is an odd prime $q$ such that $n=qs$. 
There are two cases to consider: 1. $q\not |s$ and 2. $q|s$. We will handle each 
of these cases in turn, using Proposition \ref{prop1}.   

{\bf Case 1.} We have $2p=(q-1)\phi (s)$ which implies that $q-1|2p$. The 
only divisors of $2p$ are $1$, $2$, $p$ and $2p$. As $q\neq 2$, the possible 
values for $q-1$ are 2, $p$ or $2p$ and hence for $q$ are 3, $p+1$ or $2p+1$. However, $p+1$ is not possible, because $p+1$ is even and hence not prime. If $q=3$, then $\phi (s)=p$. As $p$ is an odd number greater than 1, this is not possible. It follows that $q=2p+1$ and so $2p+1$ is prime.

{\bf Case 2.} Here we have $2p=q\phi (s)$ and so $q|2p$. The possible 
values of $q$ are 1, 2 , $p$ or $2p$. However, as $q$ is an odd prime, we must 
have $q=p$. This implies that $\phi (s)=2$. Also, $q|s$ and so $q-1|\phi (s)$. 
It follows that $p=q=3$ and hence $2p+1=7$, a prime number.

As in both cases $2p+1$ is prime, we have proved the result.
\end{proof}

\begin{remark} A prime number of the form $2p+1$, with $p$ prime, is
said to be a safe prime. In this case, the prime number $p$ is said to be a 
Sophie Germain prime. Notice that if $q$ is a safe prime et $q\neq 5$, then $q\equiv 3(\mod 4)$. (If $q\equiv 1(\mod 4)$, then $2p\equiv 0(\mod 4)$, which implies that $2|p$, which is of course impossible.)
\end{remark}

\begin{remark} The theorem does not generalize to odd numbers. Certainly,
if $s$ is odd and $2s+1$ is prime, then $2s\in \Img \phi$. However, it may be so
that $2s+1$ is not prime and $2s\in \Img\phi$. For example, $54=\phi (81)$ and
so $54\in\Img\phi$. However, $54=2\cdot 27$ and $2\cdot 27+1=55$, which is not
prime.
\end{remark}

\begin {remark} Let us now consider $\phi ^{-1}(2p)$. From Theorem 
\ref{thm3a}, the set $\phi ^{-1}(2p)$ is empty if $2p+1$ is not prime. If 
$2p+1$ is prime, then we have
$$
A(2p) = 2p\cdot \frac{2}{1}\cdot\frac{3}{2}\cdot\frac{2p+1}{2p} = 6p+3.
$$
Hence, if $n\in \phi ^{-1}(2p)$, then $2p<n\leq 6p+3$. It is worth noticing that
$$
\phi (6p+3) = \phi (3)\phi (2p+1) = 2\cdot 2p = 4p.
$$
We will study the set $\phi ^{-1}(2p)$ in more detail further on.
\end{remark}

\begin{corollary} If $2p\in \Img \phi$, then $2^kp\in \Img \phi$, for $k\geq 1$.
\end{corollary} 

\begin{proof} For $k=1$ the result is already proved, so suppose that
$k\geq 2$. As $2p+1$ is an odd prime, $2^k$ and $2p+1$ are coprime. Therefore
$$
\phi (2^k(2p+1))=\phi (2^k)\phi (2p+1)=2^{k-1}2p=2^kp.
$$
This ends the proof.
\end{proof}

\begin{remark} It may be so that $2p+1$ is not prime but that $2^lp+1$ is prime for some $l\geq 2$. For example, $2.7+1$ is not prime, but $2^2.7+1$ is prime, or $2.17+1$ and $2^2.17+1$ are not prime, but $2^3.17+1$ is prime. In this case $2^kp\in \Img \phi$, for all $k\geq l$. It is natural to ask whether, for any prime $p$, there is a power $2^k$ of $2$ such that $2^kp+1$ is also a prime. This not the case and in fact there is an infinite number of primes $p$, for which there is no such $k$ (see \cite{mendelsohn}). 
\end{remark}

It is interesting to notice that no number of the form $2p+1$, with $p$ prime, is a Carmichael number. This is a simple application of Korselt's criterion (see, for example, \cite{childs}). If $m=2p+1$ is prime, then, by definition, $m$ is not a Carmichael number. If $m$ is composite and a Carmichael number, then for any prime $q$ dividing $m$, the number $q-1$ divides $m-1=2p$. This means that $q\in \{2,3,p+1,2p+1\}$. As $m$ is odd and $p+1$, $2p+1$ are not prime, $q=3$. However, a Carmichael number is square-free and a product of at least three distinct primes. Therefore $m$ is not a Carmichael number.

\section{Sets of elements of the form $2p$} In this section we will need an
important theorem due to Dirichlet. Chapman \cite{chapman} has recently given a relatively elementary proof of this result. 

\begin{theorem} (Dirichlet) If $n\in \N^*$ and $\gcd (a,n)=1$, then there is an infinite number of prime numbers $p$ such that 
$$
p\equiv a (\mod n ).
$$
\end{theorem}
 
Let us consider the prime numbers $p$ in the interval $[1,50]$. There are $15$ such
numbers, namely 
$$
2,3,5,7,11,13,17,19,23,29,31,37,41,43,47.
$$ 
For seven of these numbers, $2p+1$ is prime, i.e. $2,3,5,11,23, 29,41$, and for
the others $2p+1$ is not prime. Therefore
$$
4,6,10,22,46,58,82\in \Img \phi \qquad\mathrm{and}\qquad 14,26,34,38,62,74,86,94\notin \Img
\phi .
$$

It is natural to ask whether there is an infinite number of distinct primes $p$
such that $2p\in \Img \phi$ (resp. $2p\notin \Img \phi$). Our question amounts 
to asking whether there is an infinite number of safe primes, or equivalently an infinite number of Sophie Germain primes. Up till now this question has not been answered. The largest known safe prime, at least up to January, 2007, was found by David Undebakke and is $48047305725.2^{172404}-1$. We can say a lot more concerning numbers of the form $2p$ which are not in the image of $\phi$. 

\begin{theorem} For any odd prime number $p$, there is an infinite set $S(p)$ of  prime numbers $q$ such that $2q\notin \Img \phi$ and $p|2q+1$.
\end{theorem}

\begin{proof} If $p$ is an odd prime, then $\frac{p-1}{2}$ is a
positive integer and $\gcd (\frac{p-1}{2},p)=1$. From Dirichlet's theorem, we know
that there is an infinite number of prime numbers of the form
$q=\frac{p-1}{2}+kp$, with $k\in \Z$. Then 
$$
2q+1 = 2\left(\frac{p-1}{2}+kp\right)+1 = p+2kp = p(1+2k).
$$
From Theorem \ref{thm3a}, we know that $2q\notin \Img \phi$ and clearly
$p|2q+1$.
\end{proof}

\begin{remark} The sets $S(p)$ may have common elements, but in general
are distinct. For example, $14$ is in $S(3)$ and
$S(5)$, but not in $S(7)$, $26$ is in $S(3)$, but not in $S(5)$ and $S(7)$ and
$34$ is in $S(5)$ and $S(7)$, but not in $S(3)$.
\end{remark}

The following result follows directly from the theorem.

\begin{corollary} \label{cor4a} For every odd prime number $p$, there exists an
infinite subset $\tilde{S}(p)$ of $\N^*$ such that
$$
a\equiv -1 (\mod p).
$$
and $a\notin \Img \phi$, when $a\in \tilde{S}(p)$.
\end{corollary}

At least for the moment, we cannot say that there is an infinity of odd prime numbers $p$ such $2p\in \Img \phi$. However, we can say that that there is an infinity of odd numbers $s$ such that $2s\in \Img \phi$. Here is a proof. From Dirichlet's theorem, there is an infinity of prime numbers $p$ such that 
$p\equiv 3 (\mod 4)$. Thus, for each of these primes, there exists 
$k\in \Z$ such that 
$$
p = 3 + 4k = 1 + 2(1+2k).
$$
As $1 + 2(1+2k)$ is prime $2(1+2k)\in \Img \phi$ and of course $1+2k$ is odd.

\section{Factorials} Up to here we have concentrated on numbers of the form
$2p$, with $p$ prime. We may consider other subsets of $\N$. One such subset is
that composed of numbers of the form $m=n!$.

\begin{lemma} If $n\geq 3$ and $p_1, \ldots ,p_s$ are the odd prime 
numbers less than or equal to $n$, then there there are integers 
$\alpha \in \N$ and $\alpha _1,\ldots ,\alpha _s\in \N^*$ such that
$$
n! = 2^{\alpha}p_1^{\alpha _1}(p_1-1)\ldots p_s^{\alpha _s}(p_s-1).
$$
\end{lemma}

\begin{proof} We can write 
$$
n! = \prod _{i=1}^s(p_i-1)p_i\prod m_j,
$$
where the $m_j$ are positive integers, with $m_j\leq n$ and $m_j\neq p_i$ and
$m_j\neq p_i-1$ for any $i$. Each $m_j$ can be written
$$
m_j = 2^{\beta}p_1^{\beta _1}\ldots p_s^{\beta _s},
$$
where $\beta ,\beta _1,\ldots ,\beta _s\in \N$. The result now follows.
\end{proof}

\begin{theorem} For all $n\in\N$, $n!$ is in the image of $\phi$.
\end{theorem}

\begin{proof} As $0!=1!=1$ and $2!=2$, the result is true for
$n=0,1,2$. From the lemma, for $n\geq 2$,
$$
n! =  2^{\alpha}p_1^{\alpha _1}(p_1-1)\ldots p_s^{\alpha _s}(p_s-1)
$$
and so
$$
n! = \phi (2^{\alpha +1}p_1^{\alpha _1+1}\ldots p_s^{\alpha _n+1}).
$$
This finishes the proof.
\end{proof}

\section{Structure of $\phi ^{-1}(m)$} If $n$ is an odd solution of the 
equation $\phi (n)=m$, then $2n$ is also a solution (Corollary \ref{cor1a}). It
follows that the equation can have at most half its solutions odd. It is 
natural to look for cases where there are exactly the same number of odd and 
even numbers of solutions.\\

First let us consider the case where $m=2p$ and $p$ is an odd prime.
If $p=3$, then $6<n\leq 21$. A simple check shows that $\phi
^{-1}(6)=\{7,9,14,18\}$. There are two odd and two even solutions of the
equation $\phi (n)=6$.

\begin{proposition} \label{prop4a} If $p$ is a prime number such that $p\geq 5$ and $2p+1$ is 
prime, then $\phi ^{-1}(2p)$ contains exactly one odd and one even element,
namely $2p+1$ and $4p+2$.
\end{proposition}

\begin{proof} We have already seen that, if $\phi (n)=2p$, then $2p<n\leq 6p+3$. The divisors of $2p$ are $1$, $2$, $p$ and $2p$ and so the only possible prime divisors of $n$ are $2$, $3$ and $2p+1$. If $n=2p+1$ or $4p+4$, then $\phi (n)=2p$ and there can be no other multiple $n$ of $2p+1$ such that $\phi(n)=2p$. If $n$ is not a multiple of $2p+1$ and $\phi (n)=2p$, then $n$ must be of the form $n=2^{\alpha}\cdot3^{\beta}$. If $\alpha \geq 3$, then $4|2p$, which is
impossible. Also, if $\beta \geq 2$, then $3|2p$. However, this is not possible,
because $p\geq 5$. Therefore $\alpha \leq 2$ and $\beta \leq 1$. As $n>2p\geq
10$, the only possibility is $n=12$. As $\phi (12)=4$, this is also impossible.
The result now follows.
\end{proof}

We can generalize this result to odd numbers in general. We will write $O(m)$
(resp. $E(m)$) for the number of odd (resp. even) solutions of the equation 
$\phi (n)=m$. 

\begin{theorem} \label{thm4a} If $s$ is odd and $s\geq 3$, then $O(2s)=E(2s)$. 
\end{theorem} 

\begin{proof} If the equation $\phi (n)=2s$ has no solution, then 
there is nothing to prove, so suppose that this is not the case. If $n$ is an 
odd solution of the equation, then $2n$ is also a solution. Hence $O(2s)\leq E(2s)$.\\
If $n$ is an even solution, then we may write $n=2^{\alpha}t$, with $\alpha \geq 1$ and $t$ odd. If $t=1$, then $\phi (n)$ is a power of $2$, which is not possible. It follows that $t\geq 3$. If we now suppose that $\alpha >1$, then $\phi (n)= 2^{\alpha -1}\phi (t)$ and, as $\phi (t)$ is even, $4$ divides $\phi (n)=2s$, which is not possible. Therefore $\alpha =1$ and so $n=2t$. As $\phi (t)=\phi (2t)=\phi (n)$, we must have $O(2s)\geq E(2s)$ and the result now follows.
\end{proof}

\begin{corollary} There is an infinity of numbers $m$ such that $\phi ^{-1}(m)$
is non-empty and composed of an equal number of odd and even numbers.
\end{corollary}

\begin{proof} We have seen that there is an infinite number of primes $p$ such that $p=1+2s$, with $s$ odd. The result now follows.
\end{proof}

\begin{remark} Simple calculations show that $\phi ^{-1}(28)=\{29, 58\}$. This shows that the converses of Proposition \ref{prop4a} and Theorem \ref{thm4a} are not true.
\end{remark}

At the other extreme is the case where $\phi ^{-1}(m)$ is non-empty and
composed entirely of even numbers. In considering this question, the following
result due to Gupta \cite{gupta}, is useful. We will give a modified proof of
it.

\begin{theorem} \label{th5a} $O(2^k)=0$ or $O(2^k)=1$. 
\end{theorem}

\begin{proof} As $\phi ^{-1}(1)=\{1,2\}$ and $\phi ^{-1}(2)=\{3,4,6\}$
the result is true for $k=1$ and $k=2$. Suppose now that $k>1$ and that 
$\phi (n)=2^k$, with $n$ odd. If $p$ is an odd prime such that $p-1$ divides 
$2^k$, then $p$ must be of the form $2^i+1$. The only primes of this form are 
Fermat numbers, hence if $n=p_0^{\alpha _1}\ldots p_r^{\alpha _r}$ is the 
decomposition of $n$ as a product of primes, then each $p_i$ must be a Fermat 
number. If we allow $\alpha _i=0$, then we may suppose that $p_i=F_i$. Thus
$$
2^k = \phi (F_0^{\alpha _0})\ldots \phi (F_r^{\alpha _r})=\prod _{\alpha _i\geq
1}F_i^{\alpha _i-1}(F_i-1).
$$
If $\alpha _i>1$ for some $i$, then the product is not a power of $2$, hence we
must have $\alpha _i=0$ or $\alpha _i=1$, for all $i$, and so 
$$
2^k = \alpha _02^{2^0}\ldots \alpha _r2^{2^r}.
$$
Clearly $\alpha _0\ldots \alpha _r$ is $k$ written in binary form. Therefore 
there can be at most one odd number $n$ such that $\phi (n)=2^k$. If 
$\alpha _i=1$ only when $F_i$ is prime, then
there exists $n$ odd such that $\phi (n)=2^k$ and $O(2^k)=1$. On the other hand,
if there is an $\alpha _i$ such that $F_i$ is not prime, then there does not
exist an odd number $n$ such that $\phi (n)=2^k$.
\end{proof}

If $k<32$, then $k$ can be written in binary form as
$\alpha _0\ldots \alpha _4$. As $F_0,\ldots F_4$ are all primes, there exists an
odd number $n$ such that $\phi (n)=2^k$. However, for $2^{32}$, this is not the
case, because $F_{5}$ is not prime. Up till now no Fermat number $F_n$, with
$n\geq 5$, has found to be prime, so it would seem that, for $k\geq 5$, there is
no odd number in the set $\phi ^{-1}(2^k)$. This suggests that there is 
an infinity of numbers $m$ for which $\phi ^{-1}(m)$ is non-empty and only
composed of even numbers.\\

In considering the structure of a set $\phi ^{-1}(m)$, we may focus attention on the congruence classes modulo $m$ of elements in $\phi ^{-1}(m)$. It is natural to ask whether there is an element $n\in \phi ^{-1}(m)$ which is congruent to $1(\mod m)$. This is always the case if $m+1$ is prime, because then $m+1\in \phi ^{-1}(m)$. However, if $m+1$ is composite, then this is probably not true. If $m+1$ is composite, $n\equiv 1(\mod m)$ and $n\in \phi ^{-1}(m)$, then there exists $k$ such that
$$
n-1 = km = k\phi (n).
$$
The problem of solving this equation for $n$ composite was first considered by Lehmer \cite{lehmer}. He found no such $n$ and showed that any solution must be odd, square-free and have at least $7$ prime factors. Up till now no solution has been found and the lower bound has been pushed much higher. (A good discussion of this problem may be found in \cite{grytczuk}.) For this reason it is most likely that $\phi ^{-1}(m)$ contains no $n\equiv 1 (\mod m)$, when $m+1$ is composite.

\section{Further remarks} If $k\in \N$, then $2.3^k\in \Img \phi$, because $\phi (3^{k+1})=2.3^k$. However, if $p$ is an odd prime other than $3$, then it is not necessarily so that $2p^k\in \Img \phi$.

\begin{proposition} If $p$ is an odd prime other than $3$ and $k\in \N$, then $2p^k\in \Img \phi$ if and only if $2p^k+1$ is a prime number, $2p^k+1\equiv 3 (\mod 4)$ and $k$ is odd.
\end{proposition}

\noindent \textsc{proof} If the conditions are satisfied, then clearly $2p^k\in \Img \phi$.

Now suppose that $2p^k\in \Img \phi$. There exists a prime number $q$ and a number $l\in \N$ such that $2p^k = q^l(q-1)$. If $l\neq 0$, then $p=q$. As $p^k|p^l(p-1)$ and $p^k$ and $p-1$ are coprime, $p^k|p^l$. This implies that $2=p^{l-k}(p-1)$ and it follows that $l=k$ and $p-1=2$. However, this contradicts the fact that $p\neq 3$. We have shown that $l=0$ and hence that $2p^k+1$ is prime. 

If $2p^k+1\equiv 1 (\mod 4)$, then $2p^k\equiv 0(\mod 4)$. It follows that $2|p^k$, which is clearly impossible. Therefore $2p^k+1\equiv 3 (\mod 4)$.

If $p\equiv 1 (\mod 3)$, then $2p^k+1\equiv 0 (\mod 3)$ and so $2p^k+1$ is not prime. Hence $p\equiv 2 (\mod 3)$. However, if $k$ is even, then $2p^k+1\equiv 0 (\mod 3)$ and $2p^k+1$ is not prime. It follows that $k$ is odd.\bx\\

Clearly this proposition generalizes Theorem \ref{thm3a}.

\section{Conclusion} We have seen that the preimage $\phi ^{-1}(m)$ is an empty set for any odd positive integer $m>1$ and also for an infinite number of even positive integers. We have also seen that, when $\phi ^{-1}(m)$ is non-empty,     $\phi ^{-1}(m)$ can have at most half of its members odd; it is also possible that all its members are even. Up till now no number $m$ has been found such that $\phi ^{-1}(m)$ contains only one element. Carmichael conjectured that such a case does not exist, but this is yet to be proved (or disproved). However, for $k\geq 2$, Ford \cite{ford} has shown that there is a number $m$ such $\phi ^{-1}(m)$ contains precisely $k$ elements.\\
   
\section{Acknowlegements} 

I would like to thank Dominique Duval and Mohamed El Methni for reading the text and offering helpful suggestions. I would also like to express my thanks to Alexei Pantchichkine, with whom I have had several useful conversations.\\

\end{document}